\documentclass[a4paper,12pt]{article}

\usepackage{color}
\usepackage[pdftex]{graphicx}
\usepackage{amsmath,amssymb,amsthm}
\usepackage{wrapfig}
\usepackage{ifpdf}
\usepackage[PostScript=dvips]{diagrams}

\newtheorem{theorem}{Theorem}[section]
\newtheorem{cor}[theorem]{Corollary}

\newtheorem{df}[theorem]{Definition}

\title {Cobordism classes of maps and covers for spheres}

\author {Oleg R. Musin\thanks{The first author is partially supported by the NSF grant DMS-1400876 and the RFBR grant 15-01-99563.}\,  and Jie Wu\footnote{The second author is partially supported by the Singapore Ministry of Education research grant (AcRF Tier 1 WBS No. R-146-000-222-112) and a grant (No. 11329101) of NSFC of China.}}

\begin{document}

	\ifpdf \DeclareGraphicsExtensions{.pdf, .jpg, .tif, .mps} \else
	\DeclareGraphicsExtensions{.eps, .jpg, .mps} \fi	
	
\date{}
\maketitle

\begin{abstract}
In this paper we show that for $m>n$ the set of cobordism classes of maps from $m$--sphere to $n$--sphere is trivial. The determination of the cobordism homotopy groups of spheres  admits applications to the covers for spheres.
\end{abstract}

\medskip


\section{Introduction}

 Let $M_1$ and $M_2$ be  compact oriented manifolds of dimension $m$. Two continuous maps  $f_1:M_1\to X$  and $f_2:M_2\to X$ are called cobordant if there are a compact oriented manifold $W$ with  $\partial W=M_1 \sqcup M_2$ and  a continuous map $F:W\to X$  such that  $F|_{M_i}=f_i$ for $i=1,2$. Note that the set of cobordism classes $f:S^m\to X$ form a group $\pi^C_m(X)$ that is a quotient of $\pi_m(X)$.

In Section 2 we consider assumptions for  $X$ such that  $\pi^C_m(X)$=0  (Theorem \ref{TC}). In particular,    Corollary~\ref{thm28} states that $\pi^C_n(S^n)=\pi_n(S^n)={\mathbb Z}$ and if $m>n$ then
$$\pi^C_m(S^n)=0.$$


In Section 3 we show that for manifolds the homotopy and cobordism classes of covers are  equivalent to the homotopy and cobordism classes of their associated maps. Then we can apply results of Sections 2 for covers, in particular, see  Corollary \ref{cor36}.

\section{Cobordism classes of maps for spheres}

Consider a group of oriented cobordism classes of maps $\Omega_*^{SO}(X)$ \cite[Chapter 1]{CF}.
 Let $M_i$, $i=1,2$, be compact oriented manifolds  without boundary of dimension $m$.
Let $f_i:M_i\to X$, $i=1,2$, be  continuous maps to a space $X$.  Then $[f_1]_C=[f_2]_C$ in  $\Omega_m^{SO}(X)$ , i.e. maps $f_i$ are {\em cobordant}  if there are a compact oriented manifold $W$ with  $\partial W=M_1 \sqcup M_2$ and  a continuous map $F:W\to X$  such that  $F|_{M_i}=f_i$ for $i=1,2$.

If $M_2=\emptyset$, then  $[f_1]_C=0$. In this case  $f_1$ is called {\em null--cobordant}.

Let $M$ be a compact oriented manifold  without boundary.  We denote the set of cobordism classes of $f:M\to X$  by $[M,X]_C$.

\begin{theorem} \label{TC}
Let $X$ be a finite $CW$-complex whose integral homology $H_*(X,\mathbb{Z})$ has only 2-torsion. Let $f:S^m\to X$ be a map that  induces zero homomorphism of $m$-dimensional cohomology with coefficients in $\mathbb Z$ and $\mathbb Z_2$. Then $f$ is null-cobordant. in $\Omega_m^{SO}(X)$ the image of $f$ is 0. In particular,  $[S^m,X]_C=0$ if $\dim X<m$.
\end{theorem}
\begin{proof}
By~\cite[Theorem 17.6]{CF}, the cobordism class of $f\colon S^m\to X$ is determined by the Pontrjagin numbers and the Stiefel-Whitney numbers of the map $f$. From the definition, the Pontrjagin numbers and the Stiefel-Whitney numbers of the map $f$ are determined by its induced homomorphisms on cohomology with coefficients in $\mathbb Z$ and $\mathbb Z_2$, respectively.   The hypothesis in the statement guarantees that $f$ and the constant map induce the same homomorphism on cohomology with coefficients in $\mathbb Z$ and $\mathbb Z_2$, and hence the result.
\end{proof}


\medskip

Let $M$ be an $m$--dimensional sphere ${S}^m$.  In this case denote $[M,S^n]_C$ by $\pi_m^C(S^n)$. It is easy to prove that the cobordism classes  $\pi_m^C(S^n)$ form a group. Moreover, there is a subgroup $N$ in  $\pi_m(S^n)$ such that
$$ \pi_m^C(S^n)=\pi_m(S^n)/N.$$

\begin{cor} \label{thm28}  If $m\ne n$, then $\pi_m^C(S^n)=0$,  otherwise $\pi_n^C(S^m)={\mathbb Z}$.
\end{cor}
\begin{proof} We obviously have the case $m<n$. Theorem \ref{TC} yields the most complicated case. 

Let $m=n$. The Hopf degree theorem (see \cite[Sect. 7]{Milnor69}) states that two continuous  maps $f_1, f_2 : S^n\to S^n$ are homotopic, i.e. $[f_1]=[f_2]$ in $\pi_n(S^n)$, if and only if $\deg{f_1} = \deg{f_2}$.
It is clear that $[f_1]=[f_2]$ implies $[f_1]_C=[f_2]_C$. Now we show that from $[f_1]_C=[f_2]_C$ follows $\deg{f_1} = \deg{f_2}$. Indeed,  then we have $F:W \to S^n$ with $F|_{M_i=S^n}=f_i$.
Note that $Z:=F^{-1}(x)$ for a regular $x\in S^n$,  is a manifold of dimension one. It is easy to see that a cobordism $(Z,Z_1,Z_2)$, where $Z_i:=f_i^{-1}(x)$,  implies $\deg{f_1}=\deg{f_2}$. Thus, $\pi_n^C(S^n)=\pi_n(S^n)={\mathbb Z}$.
\end{proof}

Corollary \ref{thm28} states that $f:S^m\to S^n$ is null-cobordant for $m>n$. Therefore, we have the following result.

\begin{cor}\label{cor36} Let $m>n$. Then for any continuous map  $f:S^m\to S^n$ there are a compact oriented manifold $W$ with  $\partial W=S^m$  and  a continuous map $F:W\to S^n$   such that $F$ on the boundary coincides with $f$.
\end{cor}

\noindent\textbf{Remark.} In the earlier version of this paper, we had a proof that $\pi_m^C(S^n)=0$, where $m>n$ only for particular cases. We formulated this statement as a conjecture and sent the preprint to several topologists. Soon, Diarmuid Crowley sent us a sketch of the proof of this conjecture. Later, Alexey Volovikov pointed out to us that  Theorem \ref{TC} follows easily from \cite[Theorem 17.6]{CF}.

\section{Homotopy and cobordism classes of covers}

For open (or closed) covers $\mathcal U$ of a normal space $T$ we considered certain homotopy classes $[f_\mathcal U]$ in $[T,S^n]$ defined in \cite{MusH}.
In this section we define a homotopy equivalence for covers and prove  that two covers $\mathcal U_1$ and $\mathcal U_2$ are homotopy equivalent if and only if $[f_{\mathcal U_1}]=[f_{\mathcal U_2}]$ in $[T,{S}^n]$ (Theorem \ref{th32}).  We also prove that two covers on manifolds of the same dimension are  cobordant if and only if the corresponding cobordism classes $[f_{\mathcal U_i}]_C$ in $\Omega_*(S^n)$ are equal (Theorem \ref{thm33}).

 The homotopy invariants $[f_\mathcal U]$ can be considered as obstructions for extending covers of a subspace $A \subset X$ to a cover of all of $X$.  (Note that the classical obstruction theory (see \cite{Hu,Span}) considers homotopy invariants that equal zero if a map can be extended from the $k$--skeleton of $X$ to the $(k+1)$--skeleton and are non-zero otherwise.)
In our papers \cite{MusH,Mus17} using these obstructions we obtain generalizations of the classic KKM (Knaster--Kuratowski--Mazurkiewicz) and Sperner lemmas \cite{KKM,Sperner}.

Let $X$ be any compact oriented manifold of dimension $(m+1)$ and $A=\partial X$ be its boundary.
Let  $\mathcal U=\{U_0,\ldots,U_{n+1}\}$ be a cover of
$A$ such that the intersection of all subsets $U_i$ is empty. Then $[\mathcal U]\in[A,S^n]$, where the homotopy class $[\mathcal{U}]$ is defined in \cite{MusH}.  In the case $m=n$ we have $[A,S^n]=\mathbb Z$ and, if $[\mathcal U]\ne0$, then for any extension of this cover to a cover $\mathcal V=\{V_0,\ldots,V_{n+1}\}$  of $X$ the intersection
$$
\bigcap\limits_{i=0}^{n+1}{V_i}\ne\emptyset \eqno (4.1)
$$
This fact is a generalization of the Sperner--KKM lemma \cite[Theorem 2.6]{MusH}.

Another generalization of the KKM lemma is the following (see \cite[Corollary 3.1]{MusH}):  Let $X$ is an $(m+1)$--disc and $A=S^m$. If $[\mathcal U]\ne0$ in $\pi_m(S^n)$, then we have property $(4.1)$.

However, for $m>n$ not all pairs $(X,A)$ satisfy property $(4.1)$. For instance, $X=\mathbb{C}\mathrm{P}^2\smallsetminus\mathrm{Int}(D^4)$ and $A:=\partial{X}=S^3$. Then the Hopf map $f:S^3\to S^2$ can be extended to a continuous map $F:X\to S^2$. It implies that a coresponding cover  $\mathcal U=\{U_0,U_1,U_2,U_3\}$ can be extened to $X$ such that the intersection of all $U_i$ is empty.

\medskip

Let $\mathcal U=\{U_0,\ldots,U_{n+1}\}$ be  a collection of open sets whose union contains a normal space $T$. In other words, $\mathcal U$ is a cover of $T$.  Let  $\Phi=\{\varphi_0,\ldots,\varphi_{n+1}\}$ be a partition of unity subordinate to $\mathcal U$.   Let
$$
f_{\mathcal U,\Phi}(x):=\sum\limits_{i=0}^{n+1}{\varphi_i(x)v_i},
$$
where $v_0,\ldots,v_{n+1}$ are vertices of  an $(n+1)$--simplex $\Delta^{n+1}$ in ${\mathbb R}^{n+1}$.

Suppose the intersection of all $U_i$ is empty.
Then $f_{\mathcal U,\Phi}$ is a continuous map from $T$ to ${S}^{n}$. In \cite[Lemmas 2.1 and 2.2]{MusH} we proved that a homotopy class $[f_{\mathcal U,\Phi}]$ in $[T,{S}^{n}]$ does not depend on $\Phi$. We denote it by $[f_\mathcal U]$.

 In fact, see  \cite[Lemma 2.4]{MusH}, the homotopy classes $[f_\mathcal U]$ of covers are also well defined for closed sets. We call a family of sets $\mathcal S=\{S_0,\ldots,S_{n+1}\}$ a {\em cover} of a space $T$ if $\mathcal S$ is either an open or closed cover of $T$.

Homotopy invariants of covers we defined through homotopy invariants of maps. Let us define them directly for covers.

\begin{df} \label{def31} Let $\mathcal S_i=\{S_0^i,\ldots,S_{n+1}^i\}$, $i=1,2$,   be  covers of a normal space $T$ such that for $i=1,2$ the intersection of all subsets in $\mathcal S_i$ is empty.    We say that $\mathcal S_1$ is homotopic to $\mathcal S_2$ and write $[\mathcal S_1]=[\mathcal S_2]$ if \, $T\times[0,1]$ can be covered  by $\mathcal Q=\{Q_0,\ldots,Q_{n+1}\}$ such that $\mathcal Q$ is an extension of $\mathcal S_1 \cup \mathcal S_2$ of $T\times\{0,1\}$ and the intersection of all $Q_k$ is empty.
\end{df}

\medskip

 The following theorem extends Theorem 2.2 in \cite{MusH}.
\begin{theorem}
\label{th32}
Let $\mathcal S_i=\{S_0^i,\ldots,S^i_{n+1}\}$, $i=1,2$, be  covers of  a normal space $T$.
Suppose 
the intersection of all the $S_j^i$ in $\mathcal S_i$ is empty.
Then $[\mathcal S_1]=[\mathcal S_2]$ if and only if $[f_{\mathcal S_1}]=[f_{\mathcal S_2}]$ in $[T,{S}^n]$.
\end{theorem}
\begin{proof} From \cite[Lemma 1.11]{MusH}  it suffices to prove the theorem for open covers.  It is clear that if   $[\mathcal S_1]=[\mathcal S_2]$ then $[f_{\mathcal S_1}]=[f_{\mathcal S_2}]$.  Now we prove the converse statement.

Suppose $[f_{\mathcal S_1}]=[f_{\mathcal S_2}]$.  Let $\Phi_i$, $i=1,2$, be any partitions of unity subordinate to $\mathcal S_i$. Then there is a homotopy
$F_{\Phi}:T\times[0,1]\to S^n$ between $f_{\mathcal S_1,\Phi_1}$ and $f_{\mathcal S_2,\Phi_2}$, where  $\Phi:=(\Phi_1,\Phi_2)$.

Consider $S^n$ as the boundary of $\Delta^{n+1}$. Let  $B_i$ be the open star of a vertex $v_i$ of $\Delta^{n+1}$. Let
$$
U_\ell(\Phi):=F_{\Phi}^{-1}(B_{\ell}), \quad U(\Phi):=\{U_0(\Phi),\ldots, U_{n+1}(\Phi)\}.
$$
Then $U(\Phi)$ is a cover of $T\times[0,1]$.

Denote by $\Pi$ the set of all pairs $\Phi:=(\Phi_1,\Phi_2)$, where $\Phi_i$ is a partition of unity subordinate to $\mathcal S_i$.  Let
$$
Q_\ell:=\bigcup\limits_{\Phi\in\Pi}{U_\ell(\Phi)}, \quad {\mathcal Q}:=\{Q_0,\ldots,Q_{n+1}\}.
$$
Then $ {\mathcal Q}$ is a cover of $T\times[0,1]$ and
$$
 {\mathcal Q}|_{T\times\{0\}}=\mathcal S_1, \quad   {\mathcal Q}|_{T\times\{1\}}=\mathcal S_2.
$$
This yields $[\mathcal S_1]=[\mathcal S_2]$.
\end{proof}

\begin{df} Let $M_i$, $i=1,2$, be compact oriented manifolds  without boundary with $\dim{M_1}=\dim{M_2}$.
Let $\mathcal S_i=\{S_0^i,\ldots,S^i_{n+1}\}$, $i=1,2$, be  covers of $M_i$ such that for $i=1,2$ the intersection of all subsets in $\mathcal S_i$ is empty. We say that $\mathcal S_1$ is cobordant to $\mathcal S_2$ and write $[\mathcal S_1]_C=[\mathcal S_2]_C$ if there are a compact oriented manifold $W$ with  $\partial W=M_1 \sqcup M_2$ and its cover $\mathcal Q=\{Q_1,\ldots,Q_n\}$ such that  $\mathcal Q|_{M_i}=\mathcal S_i$, $i=1,2$, and the intersection of all $Q_k$ is empty. If $M_2=\emptyset$, then we say that  $\mathcal S_1$ is null--cobordant and write $[\mathcal S_1]_C=0$.
\end{df}

Note that if $[\mathcal S_1]_C=[\mathcal S_2]_C$, then $[f_{\mathcal S_1}]_C=[f_{\mathcal S_2}]_C$ in $\Omega_*^{SO}(S^n)$, where  for a continuous  $f:M\to S^n$  by $[f]_C$  we denote the correspondent cobordism class.

\begin{theorem} \label{thm33} Let $M_1$ and $M_2$ be compact oriented homotopy equivalent manifolds  without boundary.  Let $\mathcal S_i$, $i=1,2$, be  covers of $M_i$ such that  the intersection of all covers in $\mathcal S_i$ is empty.  Then $[\mathcal S_1]_C=[\mathcal S_2]_C$ if and only if $[f_{\mathcal S_1}]_C=[f_{\mathcal S_2}]_C$.
\end{theorem}
\begin{proof}
By definition if $[f_{\mathcal S_1,\Phi_1}]_C= [f_{\mathcal S_2,\Phi_2}]_C$, then there is a  map
$F_{\Phi}:W\to S^n$ such that $F|_{M_i}=f_{\mathcal S_i,\Phi_i}$.  Actually, the theorem can be proved by the same arguments as Theorem \ref{th32} if we substitute $T\times[0,1]$ by a cobordism $W.$
\end{proof}

From this theorem it is easy prove the following corollary.

\begin{cor} \label{cor35} Let $\mathcal S$ be a cover of a compact oriented manifold $M$ such that the intersection of all subsets in $\mathcal S$ is empty. Suppose $[\mathcal S]_C=0$. Then there is a compact oriented manifold $W$ with  $\partial W=M$  such that $\mathcal S$  can be extended to a cover $\mathcal Q$ of $W$  (i.e. $\mathcal Q|_{M}=\mathcal S$) with the empty intersection of all subsets $Q_k$.
\end{cor}


Theorem \ref{thm33} and Corollary \ref{thm28} yield

\begin{cor}\label{cor36} Let $m>n$. Then for any cover  $\mathcal U=\{U_0,\ldots,U_{n+1}\}$ of $S^m$ with the empty intersection of all subsets in $\mathcal U$  there are a compact oriented manifold $W$ with  $\partial W=S^m$  and  a cover $\mathcal Q$ of $W$   such that $\mathcal Q$ is an extension of  $\mathcal U$ with the empty intersection of all subsets in $\mathcal Q$.
\end{cor}

\noindent{\bf Acknowledgement.} We wish to thank Diarmuid Crowley, Alexander Dranishnikov, Roman Karasev,  Arkadiy Skopenkov and Alexey Volovikov for helpful discussions and comments.

 \medskip


\noindent Oleg Musin\\
 School of Mathematical and Statistical Sciences\\  University of Texas Rio Grande Valley\\ One West University Boulevard, Brownsville, TX, 78520, USA \\
{\it E-mail address:} oleg.musin@utrgv.edu

\medskip

\noindent Jie Wu\\
Department of Mathematics\\
National University of Singapore\\
S17-06-02, 10 Lower Kent Ridge Road\\
Singapore 119076\\
{\it Email address:} matwuj@nus.edu.sg\\

\end{document}